\documentclass[12pt]{article}
\newtheorem{theorem}{Theorem}[section]
\newtheorem{prop}[theorem]{Proposition}
\newtheorem{unnumber}{}

\newtheorem{prepf}[unnumber]{Proof}
\newenvironment{proof}{\prepf\rm}{\endprepf}

\begin{document}
\title{A new look at twin reduction}
\author{Peter J. Cameron\\University of St Andrews}
\date{}
\maketitle
\begin{abstract}
Twin reduction defines an equivalence relation on the vertex set of a graph,
I give a characterisation of this equivalence relation. A consequence is a
structure theorem for the automorphism group of the graph.
\end{abstract}

\section{Twins and twin reduction}

Two vertices $v$ and $w$ of a graph $\Gamma$ are \emph{twins} if they have the
same neighbours except possibly for one another. Sometimes we distinguish
between \emph{open twins} (where $v$ and $w$ are nonadjacent and
$\Gamma(v)=\Gamma(w)$) and \emph{closed twins} (where $v$ and $w$ are
adjacent and $\{v\}\cup\Gamma(v)=\{w\}\cup\Gamma(w)$). (Here $\Gamma(v)$
denotes the set of neighbours of $v$.) Note that all our graphs are finite
and simple.

A vertex $v$ cannot have both an open and a closed twin. For, if $u$ is an
open twin of $v$, and $w$ a closed twin, then $w$ is joined to $v$ and hence
to $u$, but also $w$ is not joined to $u$ and hence not to $v$, a contradiction.
So the classes of vertices which are mutual twins (together with singleton sets
consisting of non-twins) form a partition of the vertex set $V\Gamma$. Moreover,
twins can be exchanged by an automorphism fixing all other vertices, so the
automorphism group contains a direct product of symmetric groups on these
twin classes.

This suggests that, for some purposes, it might be good to keep just one 
vertex from each twin class. However, the situation is really more complicated
than this, since identifying a pair of twins (or deleting one) may create new
twins, as the following example shows.

\begin{figure}[htbp]
\begin{center}
\setlength{\unitlength}{1mm}
\begin{picture}(65,10)
\multiput(10,5)(10,0){2}{\circle*{1}}
\put(10,5){\line(1,0){10}}
\put(10,5){\line(-2,1){10}}
\put(10,5){\line(-2,-1){10}}
\put(0,0){\line(0,1){10}}
\multiput(0,0)(0,10){2}{\circle*{1}}
\multiput(0,0)(0,10){2}{\circle{2}}
\put(25,5){\line(1,0){20}}
\put(35,5){\circle*{1}}
\multiput(25,5)(20,0){2}{\circle*{1}}
\multiput(25,5)(20,0){2}{\circle{2}}
\put(50,5){\line(1,0){10}}
\multiput(50,5)(10,0){2}{\circle*{1}}
\multiput(50,5)(10,0){2}{\circle{2}}
\put(65,5){\circle*{1}}
\put(65,5){\circle{2}}
\end{picture}
\end{center}
\caption{Twin reduction can create new twins}
\end{figure}
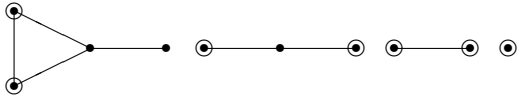

So we introduce the notion of \emph{complete twin reduction}: find a pair
of twins and delete one until no more twin remains. It is not hard to show
that the result of complete twin reduction of a graph is unique up to
isomorphism \cite[Theorem 7.1]{gog}.

The purpose of this note is to give another approach to this process. The
basic idea is that we will build a partition of the vertex set; then
contracting each part to a single vertex gives the result of complete
twin reduction. The main theorem below characterises this partition directly.

We will regard twin reduction instead as an operation on the set of partitions
of $V\Gamma$. With any such partition $\Pi$, we will associate the graph 
$\Gamma/\Pi$ obtained by shrinking each part of $\Pi$ to a single vertex. Now
begin with the partition of $V\Gamma$ into singletons. In the general step,
identify a pair of twins in $\Gamma/\Pi$, and combine the corresponding parts
of $\Pi$ into a single part. Stop when no further twins can be found.

\section{Cographs and sibling partitions}

The class of cographs has been rediscovered a number of times and has many
different characterisations~\cite{jung,seinsche,sumner}. The next result
summarises some of these, including one which is particularly relevant to twin
reduction.

\begin{theorem}
The following properties of a finite graph $\Gamma$ are equivalent.
\begin{enumerate}
\item $\Gamma$ does not contain the path of length~$3$ (with four vertices)
as induced subgraph.
\item If $\Delta$ is any induced subgraph of $\Gamma$, then either $\Delta$ or
its complement is disconnected.
\item $\Gamma$ can be constructed from $1$-vertex graphs by the operations of
complementation and disjoint union.
\item Every induced subgraph of $\Gamma$ contains a pair of twins.
\item Complete twin reduction on $\Gamma$ yields the $1$-vertex graph.
\end{enumerate}
\end{theorem}

Now we define a \emph{sibling partition} of the vertex set of a graph $\Gamma$
by the rules
\begin{enumerate}
\item the induced subgraph on any part of $\Gamma$ is a cograph;
\item between any two parts of $\Gamma$, there are either no edges or all
possible edges.
\end{enumerate}

It follows that, if $v$ and $w$ are two vertices in the same part $A$ of a
sibling partition, then $v$ and $w$ are twins in $\Gamma$ if and only if they
are twins in the induced subgraph on $A$.

\begin{prop}
Any partition obtained in the course of twin reduction is a sibling partition.
\end{prop}

\begin{proof}
The proof is by induction; clearly the partition into singletons is a sibling
partition.

Suppose that at one step of the algorithm we combine parts $A$ and $B$. The
induced subgraph on $A\cup B$ is the disjoint union of two cographs with no or
all possible edges between; the result is a cograph. Also, since $A$ and $B$
are joined to all or no vertices of each further part $C$, and are twins in the
quotient, both are joined to all of $C$, or to none of $C$. So (a) and (b) hold.
\end{proof}

\section{The main theorem}

Given two partitions $\Pi_1$ and $\Pi_2$ of a set $V$, we say that
$\Pi_1$ is \emph{finer} than $\Pi_2$ (or $\Pi_2$ is \emph{coarser} then
$\Pi_1$) if every part of $\Pi_1$ is contained in a part of $\Pi_2$

The \emph{join} of two partitions $\Pi_1$ and $\Pi_2$ of a set $V$ is the
partition $\Pi$ whose classes are connected components of the graph whose
edges are the pairs contained in a part of either $\Pi_1$ or $\Pi_2$. It is
the finest partition coarser than both $\Pi_1$ and $\Pi_2$.

\begin{theorem}
Let $\Gamma$ be a graph.
\begin{itemize}
\item[(i)] The join of any two sibling partitions is a sibling partition.
\item[(ii)] The unique maximal sibling partition is the partition produced by
complete twin reduction on $\Gamma$.
\end{itemize}
\end{theorem}

\begin{proof} (i) Let $\Pi_1$ and $\Pi_2$ be sibling partitions, and let
$\Pi=\Pi_1\vee\Pi_2$. We show that $\Pi$ satisfies (a) and (b). We begin
with (b). Let $x$ and $y$ lie in the part $A$ of $\Pi$. By definition, there
is a sequence $x=x_1,x_2,\ldots,x_m=y$ of vertices in $A$ such that each
consecutive pair is contained in a part of either $\Pi_1$ or $\Pi_2$. Now
take a vertex $z\notin A$. Since $\Pi_1$ and $\Pi_2$ satisfy (b), if
$z\sim x=x_1$, then $z\sim x_2$, \dots, and $z\sim x_m=y$. The argument is
similar if $z\not\sim x$. So (b) holds.

\smallskip

Now we prove (a). Suppose that it is false. Then there is a $4$-vertex path,
say $(a,b,c,d)$, contained in $A$. First we prove the following

\subparagraph{Claim} The points $a,b,c,d$ lie in different parts of $\Pi_1$ and
in different parts of $\Pi_2$.

Since $\Pi_1$ and $\Pi_2$ satisfy (a), the four points cannot all lie in the
same part of either.

Suppose that three of the points lie in the same part of $\Pi_1$ but the fourth
does not. Then, since $\Pi_1$ satisfies (b), the fourth point is joined to all
or none of the other three, which is clearly false. The same holds for $\Pi_2$.

Suppose that two of the points lie in the same part $A_1$ of $\Pi_1$ but the
other two do not lie in $A_1$. Then the two points not in $A_1$ are joined to
the same points in $A_1$, which (from the structure of the path) is not
possible. The same holds for $\Pi_2$. 

So the claim is proved.

\smallskip

Now $a$ and $d$ lie in the same part of $\Pi$, so there is a sequence
$a=x_1,x_2,\ldots,x_m=d$ where each consecutive pair lies in a part of
either $\Pi_1$ or $\Pi_2$. We can assume that the path of length $3$ is
chosen to minimise the length of such a sequence (which cannot be zero).
Now $a$ and $x_2$ lie in the same part of, say, $\Pi_1$ but $b,c,d$ do not
lie in this part; so $a$ and $x_2$ have the same adjacencies to $b,c,d$.
In particular, this means that $(x_2,b,c,d)$ is a path of length $3$, with
a shorter sequence joining its endpoints.

This contradiction shows that no path of length $3$ can exist in a part of
$\Pi$, so (b) is proved.

\medskip

Now we prove (ii). Let $\Pi$ be the maximal sibling partition (which is unique
by (i)). We claim that vertices
in different parts of $\Pi$ cannot be twins. For if they were, then replacing
these two parts by their union would give a coarser partition which is also a
sibling partition. (For the disjoint union and the sum of cographs are cographs;
and if $a\in A$ and $b\in B$ are twins, then they all vertices in $A\cup B$
have the same neighbours outside this set.)

So the quotient of $\Gamma$ by $\Pi$ is twin-reduced.

On the other hand, the remark before the theorem shows that we can reach $\Pi$
by twin reduction; but we cannot go any further.
\end{proof}

The two extreme cases in the theorem are:
\begin{itemize}
\item The case where $\Gamma$ has no twins. Then the maximal sibling partition
is just the partition into singletons. Note that a random graph has this
property almost surely.
\item The case where $\Gamma$ is a cograph, Then the maximal sibling partition
is the partition with a single part.
\end{itemize}

\section{Application}

\begin{theorem} 
Let $\Gamma$ be a finite graph, and $G$ its automorphism group. Then $G$ has
a well-defined normal subgroup $N$ such that
\begin{enumerate}
\item $N$ has a normal series in which every factor group is a direct product
of symmetric groups;
\item $G/N$ is a subgroup of the automorphismm group of a twin-free graph.
\end{enumerate}
\end{theorem}

\begin{proof}
(a) This is most easily seen by doing the twin reduction in stages: let
$\Pi_0$ be the partition of the vertex set into open twin classes, $\Pi_1$
the partition of the vertex set of $\Gamma/\Pi_0$ into closed twin classes,
and so on. Then the subgroup $N_0$ fixing the parts of $\Pi_0$ is a direct 
product of symmetric groups; the subgroup $N_1/N_0$ of $G/N_0$ fixing the 
parts of $\Pi_1$ is a direct product of symmetric groups; and so on.

\medskip

(b) Clear, since the quotient of $\Gamma$ by the maximal subling partition is
twin-free.
\end{proof}

Again, there are extremal cases: $N=\{1\}$ if $\Gamma$ is twin-free, and
$N=G$ if $\Gamma$ is a  cograph.

\bigskip

Graphs associated with finite groups,, such as the commuting graph, the
generating graph, and the power graph always have twin vertices. This
motivates two questions:
\begin{enumerate}
\item For which groups is one of these graphs a cograph?
\item If such a graph is not a cograph, do we get an interesting graph by
twin reduction, and how is its automorphism group related to the automorphism
group of the group we started with?
\end{enumerate}
For example, in \cite{diff}, the authors show that, if we take the graph
whose edge set if the difference of the power graph and the enhanced graph,
and the group to be the Mathieu group $M_{11}$, twin reduction produces a
semiregular bipartite graph with degrees $3$ and $4$ on $385$ vertices which
has girth~$10$ and automorphism group $M_{11}$.

\paragraph{Acknowledgement} The idea behind this paper came on an evening walk
round the beautiful campus of IIT Madras with Arun Kumar, following the ICDM
in CUSAT, Kochi. I am grateful to ADMA and IIT Madras for support and
hospitality.

\end{document}